\documentclass[a4paper,12pt]{article}

\usepackage{amsfonts}
\usepackage{amscd,color}
\usepackage{amsmath,amsfonts,amssymb,amscd}
\usepackage{indentfirst,epsfig}
\usepackage{graphicx}
\input{epsf}
\usepackage{epstopdf}
\usepackage{caption}

\newtheorem{thm}{Theorem}[section]

\newtheorem{conj}[thm]{Conjecture}
\newtheorem{lem}[thm]{Lemma}

\newenvironment {proof} {\noindent{\em Proof.}}{\hspace*{\fill}$\Box$\par\vspace{4mm}}
\newcommand{\ml}{l\kern-0.55mm\char39\kern-0.3mm}

\title{\textbf{Further results on the rainbow vertex-disconnection of  graphs\footnote{Supported by NSFC No.11871034 and11531011.}}}
\author{{\small Xueliang Li, Yindi Weng } \\
{\small  Center for Combinatorics and LPMC}\\
{\small Nankai University, Tianjin 300071, China}\\
{\small Email: lxl@nankai.edu.cn, 1033174075@qq.com}\\
}
\date{}
\begin{document}
\maketitle
\begin{abstract}
Let $G$ be a nontrivial connected and vertex-colored graph. A subset $X$ of the vertex set of $G$ is called rainbow if any two vertices in $X$ have distinct colors. The graph $G$ is called \emph{rainbow vertex-disconnected} if for any two vertices $x$ and $y$
of $G$, there exists a vertex subset $S$ such that when $x$ and $y$ are nonadjacent, $S$ is rainbow and $x$ and $y$ belong to different components of $G-S$; whereas when $x$ and $y$ are adjacent,
$S+x$ or $S+y$ is rainbow and $x$ and $y$ belong to different components of $(G-xy)-S$. Such a vertex subset $S$ is called a \emph{rainbow vertex-cut} of $G$. For a connected graph $G$, the \emph{rainbow vertex-disconnection number} of $G$, denoted by $rvd(G)$, is the minimum number of colors that are needed to make $G$ rainbow vertex-disconnected.

In this paper, we obtain bounds of the rainbow vertex-disconnection number of a graph in terms of the minimum degree and maximum degree of the graph. We give a tighter upper bound for the maximum size of a graph $G$ with $rvd(G)=k$ for $k\geq\frac{n}{2}$. We then characterize the graphs of order $n$ with rainbow vertex-disconnection number $n-1$ and obtain the maximum size of a graph $G$ with $rvd(G)=n-1$. Moreover, we get a sharp threshold function for the property $rvd(G(n,p))=n$ and prove that almost all graphs $G$ have $rvd(G)=rvd(\overline{G})=n$. Finally, we obtain some Nordhaus-Gaddum-type results: $n-5\leq rvd(G)+rvd(\overline{G})\leq 2n$ and $n-1\leq rvd(G)\cdot rvd(\overline{G})\leq n^2$ for the rainbow vertex-disconnection numbers of nontrivial connected graphs $G$ and $\overline{G}$ with order $n\geq 24$.


\end{abstract}

\section{Introduction}

All graphs considered in this paper are simple, finite and undirected. Let $G=(V(G), E(G))$ be a nontrivial connected graph with vertex set $V(G)$ and edge set $E(G)$. The $order$ of $G$ is denoted
by $n=|V(G)|$ and the $size$ of $G$ is denoted by $|E(G)|$. For a vertex $v\in V$, the \emph{open neighborhood} and \emph{closed neighborhood} of $v$ in $G$ are the set $N_{G}(v)=\{u\in V(G) | uv\in E(G)\}$ and $N_{G}[v]=N_{G}(v)\cup \{v\}$, respectively.
The \emph{degree} of $v$ in $G$ is $d_{G}(v)=|N_{G}(v)|$. The minimum and maximum degree of $G$ are denoted by $\delta(G)$ and $\Delta(G)$, respectively. Let $P_n$ denote a path with order $n$. Let $V_1$, $V_2$ be two disjoint vertex subsets of $G$.
We denote the set of edges between $V_1$ and $V_2$ in $G$ by $E(V_1,V_2)$. We follow \cite{BM} for graph theoretical notation and terminology not defined here.

In \cite{GC}, Chartrand et al. firstly studied the rainbow edge-cut by introducing the concept of rainbow disconnection of graphs.
Let $G$ be a nontrivial connected and edge-colored graph. An \emph{edge-cut} of $G$ is a set $R$ of edges of $G$ such that $G-R$ is disconnected. If any two edges in $R$ have different colors, then $R$
is a \emph{rainbow cut}. A rainbow cut $R$ is called a $u$-$v$ \emph{rainbow cut} if the vertices $u$ and $v$ belong to different components of $G-R$. An edge-coloring of $G$ is a rainbow disconnection coloring if for every two distinct vertices $u$ and $v$ of $G$,
there exists a $u$-$v$ rainbow cut in $G$, separating them. The
\emph{rainbow disconnection number} $rd(G)$ of $G$ is the minimum number of colors required by a rainbow disconnection coloring of $G$.

For vertex-colorings of graphs, the authors in \cite{BCLLW}
introduced the concept of rainbow vertex-disconnection
number. They gave some applications of the rainbow vertex-disconnection
numbers of graphs. For more results on rainbow and other colored disconnections of graphs, we refer the readers to \cite{BCL, BCJLWW, BCLLW, BHL, CLLW, PL}.

For a connected and vertex-colored graph $G$, let $x$ and $y$ be two vertices of $G$. If $x$ and $y$ are nonadjacent, then an $x$-$y$
\emph{vertex-cut} is a subset $S$ of $V(G)$ such that $x$ and $y$ belong
to different components of $G-S$. If $x$ and $y$ are adjacent, then an $x$-$y$ \emph{vertex-cut} is a subset $S$ of $V(G)$ such that $x$ and $y$ belong to different components of $(G-xy)-S$. A vertex subset $S$ of
$G$ is \emph{rainbow} if no two vertices of $S$ have the same color. An $x$-$y$ \emph{rainbow vertex-cut} is an $x$-$y$ vertex-cut $S$
such that if $x$ and $y$ are nonadjacent, then $S$ is rainbow; if $x$ and $y$ are adjacent, then $S+x$ or $S+y$ is rainbow.

A vertex-colored connected graph $G$ is called \emph{rainbow vertex-disconnected} if for any two vertices $x$ and $y$ of $G$, there exists an $x$-$y$ rainbow vertex-cut. In this case, the vertex-coloring $c$ is called a \emph{rainbow vertex-disconnection coloring} of $G$. For a connected graph $G$, the \emph{rainbow vertex-disconnection number}
of $G$, denoted by $rvd(G)$, is the minimum number of colors that are needed in order to make $G$ rainbow vertex-disconnected. A rainbow vertex-disconnection coloring with $rvd(G)$ colors is called an
$rvd$-\emph{coloring} of $G$.

An \emph{injective coloring} of a graph $G$ is a vertex-coloring of $G$ such that the colors of any two vertices with a common neighbor are different. The \emph{injective chromatic number} $\chi_i(G)$ of a graph $G$ is the minimum number of colors such that $G$ has an injective coloring using this number of colors. The injective coloring was first introduced in \cite{HKS} by Hahn et al. in 2002 and originated from complexity theory \cite{SP1}.

In this paper, we study the relationships among the graph parameters:  rainbow vertex-disconnection number, injective chromatic number, minimum degree and maximum degree. We obtain the following result in Section $2$:
$$\delta(G)\leq rvd(G)\leq \chi_i(G)\leq \Delta(G)(\Delta(G)-1)+1.$$
In Section $3$ we give a tighter upper bound for the maximum size of a  graph $G$ with $rvd(G)=k$ for $k\geq \frac{n}{2}$. In Section $4$ we characterize the graphs with rainbow vertex-disconnection number
$n-1$ and obtain the maximum size of graphs $G$ with $rvd(G)=n-1$. In Section $5$ we consider the sharp threshold function of random graphs  $G(n,p)$ with $rvd(G(n,p))=n$ and obtain that almost all graphs $G$ have $rvd(G)=rvd(\overline{G})=n$. In Section $6$ we get some Nordhaus-Gaddum-type results for the rainbow vertex-disconnection number, and leave a conjecture for further study.

\section{Preliminaries}

In this section, we first introduce some known results from \cite{BCLLW}. Then we obtain some bounds for the rainbow vertex-disconnection number of a graph.

\begin{lem}\label{rvddifcolor}\cite{BCLLW}
Let $G$ be a nontrivial connected graph, and let $u$ and $v$ be two vertices of $G$ having at least two common neighbors.
Then $u$ and $v$ receive different colors in any rvd-coloring of $G$.
\end{lem}

\begin{lem}\label{rvdn}\cite{BCLLW}
Let $G$ be a nontrivial connected graph of order $n$. Then $rvd(G)=n$ if and only if any two vertices of $G$ have at least two common neighbors.
\end{lem}

\begin{thm}\label{rvdmindegree}
Let $G$ be a connected graph of order $n$ with minimum degree $\delta$. If $\delta\geq \frac{n+2}{2}$, then $rvd(G)=n$.
\end{thm}
\begin{proof}
Since $\delta\geq \frac{n+2}{2}$, there exist at least $\frac{n+2}{2}\times2-n=2$ common neighbors for any two vertices of $G$. By Lemma \ref{rvdn}, we have $rvd(G)=n$.
\end{proof}

Let $x$ and $y$ be two vertices of a graph $G$. The \emph{local
connectivity} $\kappa_G(x,y)$ of two nonadjacent vertices $x$ and $y$ is the minimum number of vertices required to separate $x$ from $y$. If $x$ and $y$ are adjacent vertices, the local connectivity $\kappa_G(x,y)$ of $x$ and $y$ is defined as $\kappa_{G-xy}(x,y)+1$. The \emph{connectivity} $\kappa(G)$ of $G$ is the minimum number of vertices of $G$ whose removal results in a disconnected graph or a trivial graph. The \emph{upper connectivity} $\kappa^+(G)$ of $G$ is the upper bound of the function $\kappa_G(x,y)$ on $G$.

\begin{lem}\label{rvdlocalconn}\cite{BCLLW}
Let $G$ be a nontrivial connected graph of order $n$. Then
$\kappa(G) \leq \kappa^+(G) \leq rvd(G) \leq n$.
\end{lem}

\begin{lem}\label{rvdminlemm}\cite{M}
Let $K$ be a complete subgraph of $G$ with $E(G-K)\neq \emptyset$. Then there exists an edge $a_1a_2\in E(G-K)$ such that $k(a_1,a_2)=min\{d(a_1),d(a_2)\}$.
\end{lem}

\begin{lem}\label{rvduper}\cite{HKS}
Let $G$ be a graph with maximum degree $\Delta$. Then, $\chi_i(G)\leq \Delta(\Delta-1)+1$.
\end{lem}

\begin{thm}\label{rvdbound}
Let $G$ be a nontrivial connected graph with maximum degree $\Delta$. Then $\delta(G)\leq \kappa^+(G)\leq rvd(G)\leq \chi_i(G)\leq \Delta(\Delta-1)+1$.
\end{thm}
\begin{proof}
By Lemmas \ref{rvdlocalconn} and \ref{rvdminlemm}, we have $rvd(G)\geq \kappa^+(G)\geq \delta(G)$. Let $c$ be an injective coloring of $G$. Let $u$ and $v$ be any two vertices of $G$. Since the colors of any two vertices with a common neighbor are different under $c$, $N_{G}(u)$ is rainbow. If $u$ and $v$ are adjacent, then $N_{G}(u)\setminus\{v\}$ is a $u$-$v$ rainbow vertex-cut. If $u$ and $v$ are not adjacent, then $N_{G}(u)$ is a $u$-$v$ rainbow vertex-cut. Thus, $c$ is a rainbow vertex-disconnection coloring of $G$. By Lemma \ref{rvduper}, we have $rvd(G)\leq \chi_i(G)\leq \Delta(\Delta-1)+1$. \end{proof}

\section{Bounds on the maximum size }

In this section, we give a tighter upper bound for the maximum size of a  graph $G$ with $rvd(G)=k$ for $k\geq \frac{n}{2}$, which is better for large $k$ than that in the following lemma reported in \cite{BCLLW}.

\begin{lem}\label{rvd-max3}\cite{BCLLW}
For $k\geq 4$, let $G$ be a graph of order $n$ with rvd$(G)=k$. Then, $\frac{1}{2}k(n-1)-{k\choose2}\leq |E(G)|_{\max}\leq k(n-1)-{k\choose2}$.
\end{lem}

We need a lemma first.
\begin{lem}\label{maxedgeLem}
Let $G$ be a nontrivial connected graph with $rvd(G)=k$. Let $V_1,V_2,V_3,\cdots, V_k$ be the set of color classes of an rvd-coloring of $G$. Then for $i\in [k]$ and $|V_i|\geq 2$, we have
$$\sum_{v\in V_i} d_G(v)\leq n+\binom{|V_i|}{2}.$$
Let $S=\{v_i|v_i\in V_i\ and\ |V_i|=1\}$. We have
$$\sum_{v\in S} d_G(v)\leq (\frac{n+k}{2}-1)|S|.$$
\end{lem}
\begin{proof}
Without loss of generality, we assume that $|V_1|\leq |V_2|\leq\cdots\leq |V_k|$ and $s=|S|$. Then $S=\{v_1,v_2,\cdots,v_s\}$. For vertices $v_1$ and $v_2$, since $V_j$ ($j=s+1,s+2,\cdots,k$) is monochromatic, the vertices $v_1$ and $v_2$ have at most one common neighbor in $V_j$;  Otherwise, assume that $u_1,u_2\in V_j$ are the common neighbors of $v_1$ and $v_2$. Then we have that $v_1,v_2$ are two common neighbors of $u_1$ and $u_2$. So $u_1,u_2$ have different colors, a contradiction. So, we obtain $|E(v_1,V_j)|+|E(v_2,V_j)|\leq |V_j|+1$. Then we have
\begin{align*}
d_G(v_1)+d_G(v_2)&= |E(v_1,S-v_1)|+\sum_{j\in\{s+1,\cdots,k\}}|E(v_1,V_j)|\\
      &+|E(v_2,S-v_2)|+\sum_{j\in \{s+1,\cdots,k\}}|E(v_2,V_j)|\\
      &\leq 2(s-1)+\sum_{j\in \{s+1,\cdots,k\}}(|V_j|+1)\\
      &=2(s-1)+n-s+k-s\\
         &=n+k-2.
\end{align*}

Since the above inequality holds for any two vertices in $S$, we can derive that
$\sum_{i\in [s]}d_G(v_i)\leq \frac{(n+k-2)s}{2}$.

Now consider the degrees of vertices in $V_j$.
Let $\widetilde{d}(v)=|E(v,V_j)|$, where $v\in V(G)-V_j$. Let $T=\{v|\widetilde{d}(v)\geq 2\}$. Since $V_j$ is monochromatic, there are $\binom{|V_j|}{2}$ pairs of vertices in $V_j$ which have at most one common neighbor. Assume that $|E(V_j)|\leq \frac{|V_i|}{2}$.
For $v\in T$, when $\widetilde{d}(v)$ increases one, this will increase at least one pair of vertices in $V_j$ which has one common neighbor $v$. Then we have
\begin{align*}
|E(V_j,V(G)-V_j)|&= |V(G)-V_j-T|+\sum_{v\in T}{\widetilde{d}(v)}\\
               &=n-|V_j|+\sum_{v\in T}(\widetilde{d}(v)-1)\\
               &\leq n-|V_j|+\binom{|V_j|}{2}.
\end{align*}
Thus, we obtain
$$\sum_{v\in V_j}d_G(v)= 2|E(V_j)|+|E(V_j,V(G)-V_j)|\leq n+\binom{|V_j|}{2}.$$

If $|E(V_j)|>\frac{|V_j|}{2}$, assume that there are $p$ connected components $T_1,T_2,\cdots, T_p$ in $V_j$, which are trees. Each $T_i$ ($i\in[p]$) has at least $|T_i|-2$ pairs of vertices which have a common neighbor in $V_j$. Since
$$\sum_{i\in [p]}(|T_i|-2)=\sum_{i\in[p]}|T_i|-2p=|V_j|-2(|V_j|-|E_j|)=2|E_j|-|V_j|,$$
we have at least $2|E_j|-|V_j|$ pairs of vertices of $V_j$ which have no common neighbor in $V(G)-V_j$. So, we have
\begin{align*}
\sum_{v\in V_j}d_G(v)&= 2|E(V_j)|+|E(V_j,V(G)-V_j)|\\
                     &\leq2|E(V_j)|+n-|V_j|+\binom{|V_j|}{2}-(2|E_j|-|V_j|)\\
                    &=n+\binom{|V_j|}{2}.
\end{align*}
\end{proof}

\begin{thm}\label{bettermaxsizebound}
Let $G$ be a nontrivial connected graph with $rvd(G)=k$ for $k\geq \frac{n}{2}$. Then $|E(G)|_{\max}\leq \frac{(n+k-2)(2k-n)}{4}+\frac{(n-k)(n+1)}{2}$.
\end{thm}
\begin{proof}
Let $V_1,V_2,V_3,\cdots, V_k$ be the set of color classes of an rvd-coloring of $G$. Assume that $S=\{v_i|v_i\in V_i\ and\ |V_i|=1\}$ and $s=|S|$. For any two $V_{j_1}$ and $V_{j_2}$ with $|V_{j_1}|\geq |V_{j_2}|\geq 3$, we move one vertex $u$ from $V_{j_2}$ to $V_{j_1}$. Then we have
\begin{align*}
&\sum_{v\in V_{j_1}\cup\{u\}} d_G(v)+\sum_{v\in V_{j_2}\setminus\{u\}} d_G(v)\\
 &\leq n+\binom{|V_{j_1}|+1}{2}+n+\binom{|V_{j_2}|-1}{2}\\
&=n+\binom{|V_{j_1}|}{2}+n+\binom{|V_{j_2}|}{2}+|V_{j_1}|-(|V_{j_2}|-1).
\end{align*}

We find the bound is larger after moving. So, there will be $k-s-1$ color classes with order $2$ and one color classes with order $n-s-2(k-s-1)=n-2k+s+2$. Now we define the upper bound function $f(s)$ as follows:
$$f(s)=\frac{(n+k-2)s}{2}+(k-s-1)(n+1)+n+\binom{n-2k+s+2}{2}.$$
Since $k-1\geq s\geq 2k-n$ and the axis of symmetry of function $f(s)$ is $x=\frac{3k-n+1}{2}$, we get the maximum value of $f(s)$ at $s=2k-n$. Since $f(2k-n)= \frac{n+k-2}{2}(2k-n)+(n-k)(n+1)$, by Lemma \ref{maxedgeLem}, we obtain $|E(G)|_{\max}\leq \frac{1}{2}\sum_{v\in V(G)}d_G(v)\leq \frac{1}{2}f(2k-n)=\frac{(n+k-2)(2k-n)}{4}+\frac{(n-k)(n+1)}{2}$. This upper bound is tighter than the upper bound $k(n-1)-{k\choose2}$ in Lemma \ref{rvd-max3} for $k\geq \frac{n}{2}$.
\end{proof}

\section{Graphs with rainbow vertex-disconnection number $n-1$}

Let $x$ and $y$ be two vertices of a graph $G$. We denote the set of common neighbors of $x$ and $y$ by $M_{G}(x,y)$. Let $m_G(x,y)=|M_G(x,y)|$. Let $S_{G}(x,y)$ be an $x$-$y$ rainbow vertex-cut in $G$. Let $D_G(x,y)$ be the rainbow vertex set such that if $x,y$ are adjacent, then $S_{G}(x,y)+x\subseteq D_G(x,y)$ or $S_{G}(x,y)+y\subseteq D_G(x,y)$ and $D_G(x,y)$ is rainbow; if $x,y$ are nonadjacent, then $S_{G}(x,y)\subseteq D_G(x,y)$ and $D_G(x,y)$ is rainbow. In order to prove that there exists an $x$-$y$ rainbow vertex-cut in $G$, we only need to find $D_G(x,y)$.

\begin{thm}\label{rvdn-1item}
Let $G$ be a nontrivial connected graph of order $n$. Then $rvd(G)=n-1$ if and only if $G$ satisfies the following three conditions:

\item{1}. There exists at least one pair $(x,y)$ of vertices with $m_G(x,y)\leq 1$.

\item{2}. For any two pairs $(x,y)$ and $(p,q)$ of vertices with $m_G(x,y)\leq 1$ and $m_G(p,q)\leq 1$, Fig. \ref{rvdn-1}.(1) or (2) is a subgraph of $G$ containing the vertex set $\{x,y,p,q\}$.

\item{3}. For any three pairs $(x,y)$, $(x,z)$, $(y,z)$ of vertices with $m_G(x,y)\leq 1$, $m_G(x,z)\leq 1$ and $m_G(y,z)\leq 1$,
    Fig. \ref{rvdn-1}.(3) or (4) is a subgraph of $G$ containing the vertex set $\{x,y,z\}$.

\begin{figure}[h]
    \centering
    \includegraphics[width=0.8\textwidth]{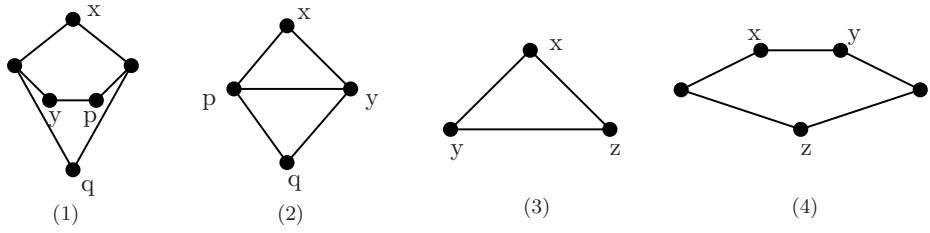}
    \caption{The graphs of condition $2$ and $3$.}\label{rvdn-1}
\end{figure}
\end{thm}
\begin{proof}
Let $rvd(G)=n-1$. Assume, to the contrary, that the graph $G$ does not satisfy at least one of the conditions. Then there are three cases to discuss.

\textbf{Case 1.} Each pair of vertices have at least two common neighbors.

By Lemma \ref{rvdn}, we have $rvd(G)=n$, a contradiction.

\textbf{Case 2.} There exist two pairs $(x,y)$ and $(p,q)$ of vertices with $m_G(x,y)\leq 1$ and $m_G(p,q)\leq 1$ which do not satisfy Condition $2$.

Define a vertex-coloring $c$ of $G$ with $n-2$ colors such that  $c(x)=c(y)=1$, $c(p)=c(q)=2$ and the remaining vertices have different colors from $3,4,\cdots, n-2$. Since $rvd(G)=n-1$, we have that $c$ is not a rainbow vertex-disconnection coloring of $G$. Then there exist two vertices $u,v$ which have no $u$-$v$ rainbow vertex-cut. Next, we claim that such vertices $u,v$ do not exist.

Let $P_1$ be the $u$-$v$ path of length two through a vertex with color $1$. Let $P_2$ be the $u$-$v$ path of length two through another vertex with color $2$. Let $P_3$ be the $u$-$v$ path of length three through two vertices with color $1$ and color $2$. Since $m_G(x,y)\leq 1$ and  $m_G(p,q)\leq 1$, there is at most one path $P_1$, at most one path $P_2$ and at most two internally disjoint paths $P_3$.

Consider that $u$ and $v$ are not adjacent. If $u\in \{x,y,p,q\}$ or $v\in\{x,y,p,q\}$, without loss of generality, assuming $u=x$, then $N_{G}(u)$ or $N_{G}(v)$ is a $u$-$v$ rainbow vertex-cut. So $u,v\not\in \{x,y,p,q\}$. There are several cases to deal with. Because of the symmetry of $P_1$ and $P_2$, some cases can be omitted. If there are no $P_1$, $P_2$ and $P_3$, then $D_{G}(u,v)=V(G)\setminus\{u,v,y,q\}$.
If there is one $P_1$ but no $P_2$, $P_3$, assuming $P_1=uxv$, then $D_{G}(u,v)=V(G)\setminus\{u,v,y,p\}$. If there is one $P_3$ but no $P_1$, $P_2$, assuming $P_3=uxpv$, then $D_{G}(u,v)=V(G)\setminus\{u,v,y,p\}$.
If there are $P_1$, $P_2$ but no $P_3$, assuming $P_1=uxv$ and $P_2=upv$, then $D_{G}(u,v)=V(G)\setminus\{u,v,y,q\}$.
If there are $P_1$, $P_3$ but no $P_2$, then assume $P_1=uxv$. When there exists one path $P_3$ which is internally disjoint with $P_1$, assuming $P_3=uypv$, we have $D_{G}(u,v)=V(G)\setminus\{u,v,y,q\}$. When all the paths $P_3$ pass the vertex $x$, since $m_G(p,q)\leq 1$, we only have one path $P_3$. Then $D_{G}(u,v)=V(G)\setminus\{u,v,y,p\}$.
If there are $P_1$, $P_2$ and $P_3$, then assume $P_1=uxv$ and $P_2=upv$. When there exists one path $P_3$ which is internally disjoint with $P_1$ and $P_2$, we have that Fig. \ref{rvdn-1}.(1) is a subgraph of $G$ containing $\{x,y,p,q\}$, a contradiction. When each path $P_3$ has a common vertex (not $u,v$) with $P_1$ or $P_2$, we have $D_{G}(u,v)=V(G)\setminus\{u,v,y,q\}$.

So, $u$ and $v$ are adjacent. When $u,v\not\in \{x,y,p,q\}$, similar to the situation where $u$ and $v$ are nonadjacent, there exists a $u$-$v$ rainbow vertex-cut. When $u\in \{x,y,p,q\}$ and $v\not\in \{x,y,p,q\}$, we have $N(u)\setminus\{v\}$ or $N(v)\setminus\{u\}$ is a $u$-$v$ rainbow vertex-cut. So, $u,v\in \{x,y,p,q\}$. If the colors of $u,v$ are the same, then $N(u)\setminus\{v\}$ or $N(v)\setminus\{u\}$ is a $u$-$v$ rainbow vertex-cut. So, the colors of $u$ and $v$ are different. Without loss of generality, we have $u=x$, $v=p$. If there is no $P_1=uyv$, then $D_{G}(u,v)=V(G)\setminus\{y,v\}$. If there is no $P_2=uqv$, then $D_{G}(u,v)=V(G)\setminus\{u,q\}$. So, there exist two paths $uyv$ and $uqv$. Thus, Fig. \ref{rvdn-1}.(2) is a subgraph of $G$ containing $\{x,y,p,q\}$, which is a contradiction.

\textbf{Case 3.} There exist three pairs $(x,y)$, $(x,z)$, $(y,z)$ of vertices with $m_G(x,y)\leq 1$, $m_G(y,z)\leq 1$ and $m_G(z,x)\leq 1$ which do not satisfy Condition $3$.

Define a vertex-coloring $c$ of $G$ with $n-2$ colors such that  $c(x)=c(y)=c(z)=1$, and the remaining vertices have different colors from $2,3,\cdots, n-2$. Since $rvd(G)=n-1$, we have that $c$ is not a rainbow vertex-disconnection coloring of $G$. Then there exist two vertices $u$ and $v$ which do not have a $u$-$v$ rainbow vertex-cut. Next, we claim that such vertices $u,v$ do not exist.

Let $Q_1$ be the $u$-$v$ path of length two through a vertex with color $1$. Let $Q_2$ be the $u$-$v$ path of length three through two vertices with color $1$. Since $m_G(x,y)\leq 1$, $m_G(y,z)\leq 1$ and $m_G(z,x)\leq 1$, there is at most one path $Q_1$ and at most one path $Q_2$.

Assume $u,v\not\in\{x,y,z\}$. Then there exist two internally disjoint paths $Q_1$ and $Q_2$. (Otherwise, if there are no paths $Q_1$ and $Q_2$, then $D_{G}(u,v)=V(G)\setminus\{y,z,u\}$; if there is a path $Q_1$ but no path $Q_2$, assuming $Q_1=uxv$, then $D_{G}(u,v)=V(G)\setminus\{y,z,u\}$; if there is a path $Q_2$ but no path $Q_1$, assuming $Q_2=uxyv$, then $D_{G}(u,v)=V(G)\setminus\{y,z,u\}$; if there exist $Q_1$ and $Q_2$, but $Q_1$ and $Q_2$ having a common vertex (not $u$,$v$), say $x$, then $D_{G}(u,v)=V(G)\setminus\{y,z,u\}$.) So, Fig. \ref{rvdn-1}.(4) is a subgraph of $G$ containing $\{x,y,z\}$, a contradiction.

Assume $u\in \{x,y,z\}$. Without loss of generality, let $u=x$. Suppose $v\not \in\{y,z\}$. If there exists $Q_1$, assuming $Q_1=uyv$, then $D_{G}(u,v)=V(G)\setminus\{z,u\}$; if there is no $Q_1$, then $D_{G}(u,v)=V(G)\setminus\{z,u\}$. So, we have $v\in \{y,z\}$. Assume $v=y$. When vertices $u$ and $v$ are not adjacent, then $V(G)\setminus \{u,v\}$ is a $u$-$v$ rainbow vertex-cut. So, vertices $u$ and $v$ are adjacent. If there is no path $Q_1$, then $D_{G}(u,v)=V(G)\setminus\{z,u\}$. If there is a path $Q_1$, then $Q_1=uzv$. So, Fig. \ref{rvdn-1}.(3) is a subgraph of $G$ containing $\{x,y,z\}$, a contradiction.

Now we are ready to show that a graph $G$ satisfying the three conditions has $rvd(G)=n-1$.

Let $c$ be any rvd-coloring of $G$. For the sake of contradiction, assume $rvd(G)\leq n-2$. If there are at least two colors which are repeated, then there exist four vertices $v_1,v_2,v_3,v_4$ with $c(v_1)=c(v_2)$ and $c(v_3)=c(v_4)$. By Lemma \ref{rvddifcolor}, we have $m_G(v_1,v_2)\leq 1$ and $m_G(v_3,v_4)\leq 1$. Then Fig. \ref{rvdn-1}.(1) or (2) is a subgraph of $G$ containing $\{v_1,v_2,v_3,v_4\}$. So, there are at least three colors for vertices $v_1,v_2,v_3$ and $v_4$, a contradiction. If there is only one color which is repeated, then there exist at least three vertices $v_1,v_2,v_3$ with $c(v_1)=c(v_2)=c(v_3)$. Similarly, we have $m_G(v_1,v_2)\leq 1$, $m_G(v_1,v_3)\leq 1$ and $m_G(v_2,v_3)\leq 1$ by Lemma \ref{rvddifcolor}. Then Fig. \ref{rvdn-1}.(3) or (4) is a subgraph of $G$ containing $\{v_1,v_2,v_3\}$. So, there are at least two colors for vertices $v_1,v_2$ and $v_3$, a contradiction. Thus, $rvd(G)\geq n-1$. By Lemma \ref{rvdn}, we have $rvd(G)\leq n-1$.
\end{proof}

\begin{thm}\label{rvdedgen-1max}
Let $G$ be a nontrivial connected graph of order $n$ with $rvd(G)=n-1$. Then $$|E(G)|_{max}=\begin{cases}
1,  & n=2,\\
\frac{1}{2}n(n-1)-n+3,  & n\geq 3.
\end{cases}$$
\end{thm}
\begin{proof}
When $n=2$, the graph is $K_2$. Consider $n\geq 3$. Since $rvd(G)=n-1$, there exists at least one pair $(x,y)$ of vertices with $m_G(x,y)\leq 1$ by Theorem \ref{rvdn-1item}. If $x$ and $y$ are adjacent, then $d_{G}(x)+d_{G}(y)\leq n+1$ and there are at least $2(n-1)-d_{G}(x)-d_{G}(y)\geq n-3$ edges which are not in $G$. If $x$ and $y$ are not adjacent, then $d_{G}(x)+d_{G}(y)\leq n-1$ and there are at least $2(n-1)-d_{G}(x)-d_{G}(y)-1\geq n-2$ edges which are not in $G$. Thus, $|E(G)|\leq \frac{1}{2}n(n-1)-n+3$. Let $H$ be a graph with $V(H)=\{v_1,v_2,\cdots,v_n\}$, which is obtained from $K_n$ by deleting edges $v_nv_i$ ($i=[n-3]$). We have $rvd(H)=n-1$ and $E(H)=\frac{1}{2}n(n-1)-n+3$.
\end{proof}

\noindent {\bf Remark:} This improves the result of Theorem \ref{bettermaxsizebound} for the case $k=n-1$, where only bounds were given.

\section{Results for random graphs}

Let $G=G(n,p)$ be the random graphs on $n$ vertices and edge probability $p$. In the study of properties of random graphs, many researchers observed that there are sharp \emph{threshold functions} for various natural graph properties. For a graph property $A$ and for a function $p=p(n)$, we say that $G(n,p)$ satisfies $A$ \emph{almost surely} if the probability that $G(n,p(n))$ satisfies $A$ tends to $1$ as $n$ tends to infinity. We say that a function $f(n)$ is a \emph{sharp threshold function} for the property $A$ if there are two positive constants $c$ and $C$ such that $G(n,cf(n))$ almost surely does not satisfy $A$ and $G(n,p)$ satisfies $A$ almost surely for all $p\geq Cf(n)$. It is well-known that all monotone graph properties have a sharp threshold function, see \cite{BTT} and \cite{FKE}. In \cite{YAYZR}, the authors obtained the sharp threshold function for the property $rc(G(n,p))\leq 2$ by proving the property that any two vertices of $G(n,p)$ have at least $2\log n$ common neighbors. By Lemmas \ref{rvdn} and \ref{rvdrandom} we can obtain Theorem \ref{rvdsharpn} immediately.
\begin{lem}\label{rvdrandom}\cite{YAYZR}
$p=\sqrt{\log n/n}$ is a sharp threshold function for the property $rc(G(n,p))\leq 2$.
\end{lem}

\begin{thm}\label{rvdsharpn}
$p=\sqrt{\log n/n}$ is a sharp threshold function for the property $rvd(G(n,p))=n$.
\end{thm}

\begin{lem}\cite{NJ}(Chernoff Bound)\label{rvdcher}
If $X$ is a binomial random variable with expectation $\mu$, and $0<\delta<1$, then
$$Pr[X<(1-\delta)\mu]\leq exp(-\frac{\delta^2\mu}{2})$$
and if $\delta>0$,
$$Pr[X>(1+\delta)\mu]\leq exp(-\frac{\delta^2\mu}{2+\delta})$$

\end{lem}

\begin{thm}\label{rvdrandom-n}
Almost all graphs $G$ have $rvd(G)=rvd(\overline{G})=n$.
\end{thm}
\begin{proof}
Consider the random graphs $G(n,\frac{1}{2})$. By Lemma \ref{rvdn}, it suffices to show that almost surely any two vertices of $G(n,\frac{1}{2})$ have at least $2$ common neighbors and two common nonadjacent vertices. Let $x$ and $y$ be two vertices of $G(n,\frac{1}{2})$. By union bound, it suffices to show that $x,y$ do not have two common neighbors or two common nonadjacent vertices with probability $o(\frac{1}{n^2})$. Let $X_A$ be the number of common neighbors of $x$ and $y$. Let $X_B$ be the number of vertices which are not adjacent to $x,y$. Let $A_i$ be the event that vertex $i$ is the common neighbor of $x$ and $y$.  Let $B_i$ be the event that $i\notin N_{G}(x)\cup N_{G}(y)$. Then $$Pr(A_i)=Pr(B_i)=\frac{1}{4},\ \ \ E(X_A)=E(X_B)=\sum_{i}Pr(A_i)=\frac{n-2}{4}.$$
By Lemma \ref{rvdcher},
$$Pr(X_A< \frac{2n-4}{9})=Pr(X_B< \frac{2n-4}{9})\leq exp(-\frac{n-2}{648}) ,$$
where $\delta_A=\delta_B=\frac{1}{9}$.  Then, when $n$ is sufficiently large, we have
$$Pr(X_A<2\ or\ X_B<2)\leq 2exp(-\frac{n-2}{648})=o(\frac{1}{n^2}).$$
So, almost all graphs $G$ have $rvd(G)=rvd(\overline{G})=n$.
\end{proof}

\section{Nordhaus-Gaddum-type results}

In this section, we study the Nordhaus-Gaddum-type problem for the rainbow vertex-disconnection number of graphs. We assume that both a  graph $G$ and its complement $\overline{G}$ are connected of order $n$. So, we have $n\geq 4$.

\begin{lem}\cite{BCLLW}\label{rvdsubgraph}
If $G$ is a nontrivial connected graph and $H$ is a connected subgraph of $G$, then $rvd(H)\leq rvd(G)$.
\end{lem}

\begin{lem}\cite{BCLLW}\label{rvdcomplete}
For an integer $n\geq 2$,
$$rvd(K_{n})=\left\{
\begin{array}{lcl}
n-1,       &      & {if~n=2,3},\\
n,         &      & {if~n\geq 4}.
\end{array} \right .$$
\end{lem}

\begin{lem}\label{rvdkn-e}
$rvd(K_n-e)=n$ for $n\geq 5$ and $rvd(K_n-2e)=n$ for $n\geq 6$.
\end{lem}
\begin{proof}
If $n=5$, then $|E(K_5-e)|=9$. By Theorem \ref{rvdedgen-1max}, we have $rvd(K_5-e)=5$. For $n\geq 6$, since $E(K_n-2e)=\frac{1}{2}n(n-1)-2$, we have $rvd(K_n-2e)=n$ by Theorem \ref{rvdedgen-1max}. By Lemma \ref{rvdsubgraph}, we have $rvd(K_n-e)\geq rvd(K_n-2e)=n$.
\end{proof}

\begin{lem}\cite{BCLLW}\label{rvd1}
Let $G$ be a nontrivial connected graph. Then $rvd(G)=1$
if and only if $G$ is a tree.
\end{lem}

\begin{lem}\label{rvdcomplemmn=57}
If $G$ is a connected graph of order $n$ with $5\leq n\leq 7$ and  $rvd(G)=1$, then $rvd(\overline{G})\geq n-2$ and the lower bound is sharp.
\end{lem}
\begin{proof}

\begin{figure}[h]
    \centering
    \includegraphics[width=0.6\textwidth]{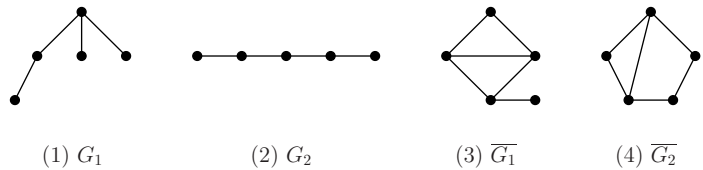}
    \caption{The trees and their complement graphs with order $n=5$.}\label{rvdNGn=5}
\end{figure}

By Lemma \ref{rvd1}, $G$ is a tree. For $n=5$, the graph $G$ is $G_1$ or $G_2$ as shown in Fig. \ref{rvdNGn=5}. Since $rvd(\overline{G_1})=rvd(\overline{G_2})=3$, we have $rvd(\overline{G})=n-2$.

\begin{figure}[h]
    \centering
    \includegraphics[width=0.9\textwidth]{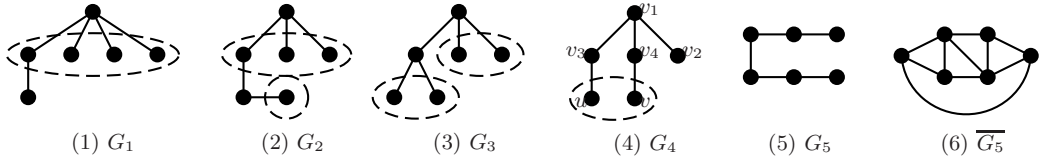}
    \caption{The trees with order $n=6$ and the complement graph of $P_6$.}\label{rvdNGn=6}
\end{figure}

For $n=6$, the graph $G$ is one of $G_1$ through $G_5$ as shown in Fig. \ref{rvdNGn=6}.(1) through (5).
For $G_1$ through $G_3$, the four vertices in dashed line cycle form a $K_4$ in $\overline{G}$. By Lemmas \ref{rvdcomplete} and \ref{rvdsubgraph}, we have $rvd(\overline{G})\geq rvd(K_4)\geq 4$. If $G$ is the graph $G_4$, then there are four internally disjoint paths between $u$ and $v$ in $\overline{G}$, which are paths $uv$, $uv_1v$, $uv_2v$ and $uv_4v_3v$. So, by Lemma \ref{rvdlocalconn} we have $rvd(\overline{G})\geq 4$. If $G$ is the graph $G_5$, then $\overline{G}$ is the graph as shown in Fig. \ref{rvdNGn=6}.(6). We have $rvd(\overline{G_5})=rvd(\overline{P_6})=4$. So, $rvd(\overline{G})\geq n-2$.

\begin{figure}[h]
    \centering
    \includegraphics[width=0.8\textwidth]{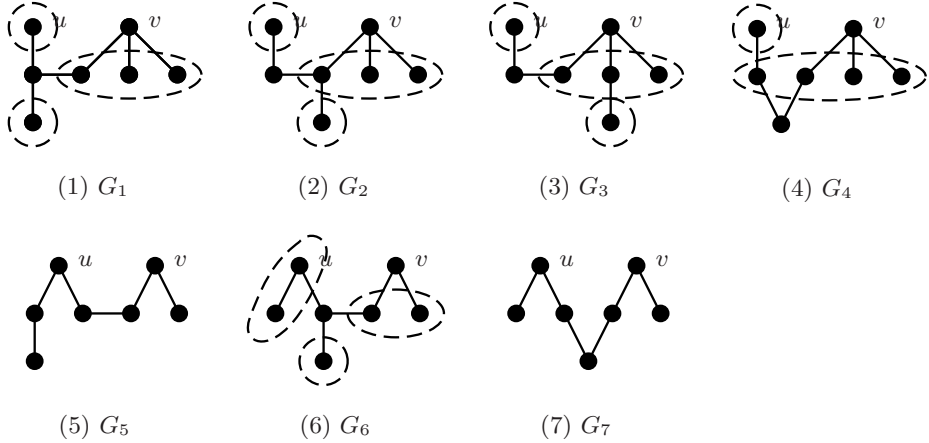}
    \caption{The trees with order $n=7$ and $|E(H)|=2$.}\label{rvdNGn=7}
\end{figure}

For $n=7$, if any two vertices of the graph $\overline{G}$ have at least two common neighbors, then $rvd(\overline{G})=n$ by Lemma \ref{rvdn}. So,  assume that there exist two vertices $x,y$ in $\overline{G}$ with $m_{\overline{G}}(x,y)\leq 1$. Then there is at most one vertex which is not adjacent to $x$ and $y$ in $G$. Let $H=G-x-y$. If $|E(H)|\geq 3$, assuming that $\{e_1,e_2,e_3\}\subseteq E(H)$, then there are at least two edges with the common vertex, say $e_1=v_1v_2$ and $e_2=v_2v_3$. Since $G$ is a tree, we have one vertex $t\in \{v_1,v_2,v_3\}$ but $t\not \in N_{G}(x)\cup N_{G}(y)$ and there is a path $P_1=xe_1y$ or $xe_2y$ or $xe_1e_2y$ in $G$. So, the endpoints of $e_3$ are adjacent to $x$ and $y$, respectively. There is a path $P_2=xe_3y$. Then according to $P_1$ and $P_2$, there is a cycle in graph $G$, a contradiction. So, $|E(H)|\leq 2$. When $|E(H)|\leq 1$, we have $\overline{H}=K_5$ or $K_5-e$. By Lemmas \ref{rvdsubgraph}, \ref{rvdcomplete} and \ref{rvdkn-e}, we have $rvd(\overline{G})\geq rvd(\overline{H})=5$. When $|E(H)|=2$, $G$ is as shown in Fig. \ref{rvdNGn=7}. If $G$ is one of $G_1$ through $G_4$ or $G_6$, then the vertices in dashed line cycle form a $K_5$ or $K_5-e$ in $\overline{G}$. So, $rvd(G)\geq 5$ by Lemmas \ref{rvdcomplete} and \ref{rvdkn-e}. If $G$ is $G_5$ or $G_7$, i.e. $P_7$, then we have $rvd(\overline{G})=rvd(\overline{P_7})=5$. \end{proof}

\begin{lem}\label{rvd1comple}
If $G$ is a connected graph of order $n\geq 8$ and $rvd(G)=1$, then $rvd(\overline{G})\geq n-1$ and the lower bound is sharp.
\end{lem}
\begin{proof}
By Lemma \ref{rvd1}, $G$ is a tree. Assume $n\geq 8$. If any two vertices of the graph $\overline{G}$ have at least two common neighbors, then $rvd(\overline{G})=n$ by Theorem \ref{rvdn}. So, we consider that there exist two vertices $x$ and $y$ in $\overline{G}$ with $m_{\overline{G}}(x,y)\leq 1$. Then there is at most one vertex which is not adjacent to $x$ and $y$ in $G$. Without loss of generality, let $N_{G}(y)\geq N_{G}(x)$. Then $|N_{G}(y)\setminus\{x\}|\geq 3$. Let $H=G-x-y$.

If $x$ and $y$ have no common nonadjacent vertex in $G$, then $V(H)\subseteq N_{G}(x)\cup N_{G}(y)$. Since there exists at most one edge in graph $H$, we have $rvd(\overline{H})=n-2$ by Lemmas \ref{rvdcomplete} and \ref{rvdkn-e}. Let $v\in V(H)$.
If $M_G(x,y)=\{s\}$, then $N_{G}(y)\geq 4$. We have $N_{G}(y)\setminus\{s,v\}\subseteq M_{\overline{G}}(x,v)$. So,  $m_{\overline{G}}(x,v)\geq 2$. By Lemma \ref{rvddifcolor}, we obtain $rvd(\overline{G})\geq n-1$.
Consider $M_G(x,y)=\emptyset$. If $xy\in E(G)$, then $N_{G}(y)\setminus \{x,v\}\subseteq M_{\overline{G}}(x,v)$. So, $m_{\overline{G}}(x,v)\geq 2$ and $rvd(\overline{G})\geq n-1$. If $xy\not \in E(G)$, then there is an edge $v_1v_2\in E(G)$ with $v_1\in N_G(x)$ and $v_2\in N_G(y)$. Then $N_{G}(y)\setminus \{v_2\}\subseteq M_{\overline{G}}(x,v)$ for $v\in N_G(x)$ and $N_{G}(y)\setminus \{v\}\subseteq M_{\overline{G}}(x,v)$ for $v\in V(H)\setminus N_G(x)$. So, $m_{\overline{G}}(x,v)\geq 2$ for $v\in V(H)$ and $rvd(\overline{G})\geq n-1$.

Now we consider that there exists a vertex $t$ which is not adjacent to $x$ and $y$ in $G$. If $xy\in E(G)$, then the vertex $t$ is adjacent to a vertex $t'$ of $N_{G}(x)\setminus\{y\}$ or $N_{G}(y)\setminus\{x\}$ in $G$. So, $rvd(\overline{H})=rvd(K_{n-2}-e)=n-2$ by Lemma \ref{rvdkn-e}. If $t'\in N_{G}(x)\setminus\{y\}$ in $G$, then $N_{G}(y)\setminus\{x,v\}\subseteq M_{\overline{G}}(x,v)$; if $t'\in N_{G}(y)\setminus\{x\}$ in $G$, then $N_{G}(y)\setminus\{x,t'\}\subseteq M_{\overline{G}}(x,v)$ for $v=t$ and $N_{G}(y)\setminus\{x,v\}\subseteq M_{\overline{G}}(x,v)$ for $v\in V(H)\setminus\{t\}$. We have $m_{\overline{G}}(x,v)\geq 2$ for $v\in V(H)$. So $rvd(\overline{G})\geq n-1$.

Assume $xy\not\in E(G)$. If $M_G(x,y)=\{s'\}$, then $t$ is adjacent to a  vertex $h_1$ from $V(H)$ in $G$. We have $rvd(\overline{H})=rvd(K_{n-2}-e)=n-2$ by Lemma \ref{rvdkn-e}. If $h_1\in N_{G}(x)$ or $N_{G}(y)\geq 4$ in $G$, then $N_{G}[y]\setminus\{s',h_1\}\subseteq M_{\overline{G}}(x,t)$, $N_{G}(y)\setminus\{s'\}\subseteq M_{\overline{G}}(x,v)$ for $v\in N_{G}(x)$ and $N_{G}(y)\cup\{t\}\setminus\{v,s'\}\subseteq M_{\overline{G}}(x,v)$ for $v\in N_{G}(y)\setminus\{s'\}$. So, by Lemma \ref{rvddifcolor} we have that $V(\overline{G})\setminus\{y\}$ is rainbow. If $h_1\in N_{G}(y)$ and $N_{G}(y)=3$ in $G$, then $N_{G}(x)=3$. By symmetry, we know that $V(\overline{G})\setminus\{x\}$ is rainbow. Thus, $rvd(\overline{G})\geq n-1$. If $M_G(x,y)=\emptyset$, then $rvd(\overline{H})=rvd(K_{n-2}-2e)=n-2$ by Lemma \ref{rvdkn-e}. Let $h_2$ be the neighbor of $t$ in $N_{G}(y)$ (if $h_2$ does not exist, then let $\{h_2\}=\emptyset$). Let $v_xv_y$ be the edge of $G$ with $v_x\in N_G(x)$ and $v_y\in N_G(y)$ (if $v_xv_y$ does not exist, then let $\{v_x\}=\{v_y\}=\emptyset$). We have $N_{G}(y)\setminus\{h_2\}\subseteq M_{\overline{G}}(x,t)$, $N_{G}(y)\setminus\{v_y\}\subseteq M_{\overline{G}}(x,v_x)$
and $N_{G}(y)\setminus\{v\}\subseteq M_{\overline{G}}(x,v)$ for $v\in V(H)\setminus\{t,v_x\}$. So, $m_{\overline{G}}(x,v)\geq 2$ for $v\in V(H)$ and $rvd(\overline{G})\geq n-1$.

For the sharpness of the lower bound, let $G_0$ be a tree obtained from $K_{1,n-2}$ by adding a new vertex which is adjacent to one of the  leaves. Then we have $rvd(\overline{G_0})=n-1$.
\end{proof}

A \emph{block} of a graph is a maximal connected induced subgraph of $G$ containing no cut vertices. A \emph{end-block} is a block with exactly one cut vertex of $G$.

\begin{lem}\cite{BCLLW}\label{rvd2}
Let $G$ be a nontrivial connected graph. Then $rvd(G)=2$
if and only if each block of $G$ is either $K_2$ or a cycle and at least one block of $G$ is a cycle.
\end{lem}

\begin{lem}\label{rvd2comple}
If $G$ is a connected graph of order $n$ and $rvd(G)=2$, then $rvd(\overline{G})\geq n-3$.
\end{lem}
\begin{proof}
If $G=C_n$, then $\delta(\overline{G})=n-3$. By Theorem \ref{rvdbound}, we have $rvd(\overline{G})\geq n-3$. Assume that $G$ is not $C_n$. By Lemma \ref{rvd2}, there exist at least two end-blocks $B_i$ and $B_j$ which are a $K_2$ or a cycle. If $B_i$ and $B_j$ are both $K_2$, then we select one of the endpoints with degree one in $G$ from $B_i,B_j$,  respectively, say $x$ and $y$. Then $x,y$ are adjacent in $\overline{G}$ and  $M_{\overline{G}}(x,y)=V(\overline{G})\setminus\{N_{G}[x],N_{G}[y]\}$. So, $\kappa_{\overline{G}}(x,y)\geq n-3$. Thus, $rvd(\overline{G})\geq n-3$ by Lemma \ref{rvdlocalconn}. Assume that $B_i$ is a cycle. If $B_i$ is a triangle $v_1v_2v_3v_1$, where $v_3$ is a cut vertex of $G$, then $M_{\overline{G}}(v_1,v_2)=V(\overline{G})\setminus\{v_1,v_2,v_3\}$. So,  $rvd(\overline{G})\geq n-3$. If $B_i$ is a $C_4=v_1v_2v_3v_4v_1$, where $v_4$ is a cut vertex of $G$, then $M_{\overline{G}}(v_1,v_3)=V(\overline{G})\setminus\{v_1,v_2,v_3,v_4\}$.
Since $v_1,v_3$ are adjacent in $\overline{G}$, we have $rvd(\overline{G})\geq \kappa_{\overline{G}}(v_1,v_3)\geq n-3$ by Lemma \ref{rvdlocalconn}. If $B_i$ is a cycle $C_t=v_1v_2\cdots v_{t-1}v_tv_1$ with order $t\geq 5$, where $v_t$ is a cut vertex of $G$, then $M_{\overline{G}}(v_1,v_2)=V(\overline{G})\setminus\{v_1,v_2,v_3,v_t\}$.
Since there is a path $v_1v_3v_tv_2$ in $\overline{G}$, we have $rvd(\overline{G})\geq \kappa_{\overline{G}}(v_1,v_2)\geq n-3$ by Lemma \ref{rvdlocalconn}.
\end{proof}

\begin{thm}\label{rvdNGmulti}
Let $G$ and $\overline{G}$ be connected graphs of order $n$.
\item 1. If $n=4$, then $rvd(G)\cdot rvd(\overline{G})=1$.
\item 2. If $5\leq n\leq 7$, then $n-2\leq rvd(G)\cdot rvd(\overline{G})\leq n^2$ and the lower bound is sharp.
\item 3. If $n\geq 8$, then $n-1\leq rvd(G)\cdot rvd(\overline{G})\leq n^2$. Furthermore, the lower bound is sharp and the upper bound is sharp for $n\geq 12$.
\end{thm}
\begin{proof}
If $n=4$, then $G=\overline{G}=P_4$. By Lemma \ref{rvd1}, $rvd(G)\cdot rvd(\overline{G})=1$. The upper bounds are obvious. For the lower bound, by Lemmas \ref{rvdcomplemmn=57} and \ref{rvd2comple}, we only need to consider $n\geq 8$. we assume that $\delta(G)\leq \delta(\overline{G})$. Then $\delta(G)\leq \frac{n-1}{2}$. When $rvd(G)=1$, we have $rvd(G)\cdot rvd(\overline{G})\geq n-1$ by Lemma \ref{rvd1comple}. When $rvd(G)=2$, we have $rvd(G)\cdot rvd(\overline{G})\geq 2(n-3)=n-1+n-5\geq n-1$ by Lemma \ref{rvd2comple}. By symmetry, we consider $rvd(G)\geq 3$ and $rvd(\overline{G})\geq 3$. So, for $rvd(G)\cdot rvd(\overline{G})$ we only need to consider $n\geq 11$. Let
$$f=
\begin{cases}
\big\lceil{\frac{n-1}{3}}\big\rceil, & if\ \delta(G)\leq 3,\\
\big\lceil{\frac{n-1}{\delta(G)}}\big\rceil, & if\ \delta(G)\geq 4.
\end{cases}$$

If $rvd(\overline{G})\geq f$, then for $\delta(G)\leq 3$, $rvd(G)\cdot rvd(\overline{G})\geq 3f\geq n-1$; for $\delta(G)\geq 4$, $rvd(G)\cdot rvd(\overline{G})\geq \delta(G)\cdot f\geq n-1$ by Theorem \ref{rvdbound}.

Suppose $rvd(\overline{G})\leq f-1$. Let $u$ be the vertex with $d_G(u)=\delta(G)$. Let $H=G-N_{G}[u]$ and $v$ be any vertex of the graph $H$. Since $u$ is adjacent to $v$ in the graph $\overline{G}$, we have $m_{\overline{G}}(u,v)\leq f-2$. So, there are at most $f-2$ vertices which are not adjacent to $u$ and $v$ in $G$. Hence, we have
\begin{align*}
d_H(v)&\geq |H|-1-(f-2)\\
      &=n-\delta(G)-f.
\end{align*}

If $H$ is connected, then $rvd(G)\geq rvd(H)\geq \delta(H)\geq n-\delta(G)-f$ by Lemma \ref{rvdsubgraph} and Theorem \ref{rvdbound}.
If $H$ is not connected, then we denote a component with maximum rainbow vertex-disconnection number by $H_1$. Then $rvd(G)\geq rvd(H_1)\geq \delta(H_1)\geq n-\delta(G)-f$ by Lemma \ref{rvdsubgraph} and Theorem \ref{rvdbound}.

When $\delta(G)\leq 3$, $rvd(G)\cdot rvd(\overline{G})\geq (n-\delta(G)-f)\cdot 3\geq n-1+n-10\geq n-1$. When $\delta(G)\geq 4$, since $\delta(G)\leq \frac{n-1}{2}$, we have $rvd(G)\cdot rvd(\overline{G})\geq (n-\delta(G)-f)\cdot \delta(G)\geq n-1+\delta(G)\cdot (\delta(G)-4)\geq n-1$.

By Lemmas \ref{rvdcomplemmn=57} and \ref{rvd1comple}, the lower bound is sharp. For the upper bound, let $G$ be a graph with order $n=4k+t$ ($k\geq 3$ and $t=0,1,2,3$). The vertex set $V(G)$ can be partitioned into four cliques, $V_1=\{v_1,v_2,\cdots,v_{k}\}$, $V_2=\{v_{k+1},v_{k+2},\cdots,v_{2k}\}$, $V_3=\{v_{2k+1},v_{2k+2},\cdots,v_{3k}\}$ and $V_4=\{v_{3k+1},v_{3k+2},\cdots,v_{4k}, v_{4k+1},\cdots,v_{4k+t}\}$. Each vertex set $\{v_i,v_{k+i},v_{2k+i},v_{3k+i}\}$ forms a clique, where $i\in[k-1]$. The vertex set $\{v_k,v_{2k},v_{3k},v_{4k},v_{4k+1},\cdots,v_{4k+t}\}$ also forms a clique. Then we have $rvd(G)=rvd(\overline{G})=n$.
\end{proof}

\begin{thm}\label{rvdNGaddn-7}
Let $G$ and $\overline{G}$ be connected graphs of order $n$. Then $n-7 \leq rvd(G)+rvd(\overline{G})\leq 2n$.
\end{thm}
\begin{proof}
The upper bound is obvious. Now we consider the lower bound. For $n\geq 8$, when $rvd(G)=1$ or $2$, we have $rvd(G)+rvd(\overline{G})\geq n-1$ by Lemmas \ref{rvd1comple} and \ref{rvd2comple}. Assume that $rvd(\overline{G})\geq rvd(G)\geq 3$. If $rvd(G)\geq \big\lceil\frac{n-1}{2}\big\rceil$, then $rvd(G)+rvd(\overline{G})\geq n-1$. So, we consider $3\leq rvd(G)\leq rvd(\overline{G})\leq\big\lceil\frac{n-1}{2}\big\rceil-1$ and $n\geq 13$. Let $rvd(G)=k$ and $\{V_1,V_2,\cdots,V_k\}$ be the set of color classes of an rvd-coloring of $G$. Since $\frac{n}{k}>2$, there are three cases to consider.

\textbf{Case 1}. There exists a $V_i$ with $|V_i|\geq 4$.

Let $D_i$ be the subset of $V_i$ with four vertices. For any two vertices of $D_i$, they have at most one common nonadjacent vertex in $\overline{G}$ by Lemma \ref{rvddifcolor}.
Let $S=\{u|$ the vertex $u$ is not adjacent to at least two vertices of $D_i$ in $\overline{G}\}$. Let $T=V(\overline{G})\setminus(D_i\cup S)$. Then $|S|\leq {4\choose 2}=6$ and $|N_{\overline{G}}(v)\cap D_i|\geq 3$ for $v\in T$. For any two vertices $x,y$ of $T$, there are at least two common neighbors from $D_i$ in $\overline{G}$. By Lemma \ref{rvddifcolor}, the vertex set $T$ is rainbow in $\overline{G}$. Thus, $rvd(\overline{G})\geq n-10$ and $rvd(G)+rvd(\overline{G})\geq n-7$.

\textbf{Case 2}. There exist $V_i$, $V_j$ with $|V_i|=|V_j|=3$ and $|V_s|\leq 3$ for $s\in [k]$.

For any two vertices of $V_i$ or $V_j$, they have at most one common nonadjacent vertex in $\overline{G}$ by Lemma \ref{rvddifcolor}.
Let $S_1=\{u|$ the vertex $u$ is not adjacent to at least two vertices of $V_i$ in $\overline{G}\}$ and $S_2=\{u|$ the vertex $u$ is not adjacent to at least two vertices of $V_j$ in $\overline{G}\}$. Let $T=V(\overline{G})\setminus(V_i\cup V_j\cup S_1\cup S_2)$. Then $|S_1\cup S_2|\leq 6$. We have $|N_{\overline{G}}(v)\cap V_i|\geq 2$ and $|N_{\overline{G}}(v)\cap V_j|\geq 2$ for $v\in T$. For any two vertices $x,y$ of $T$, $x$ and $y$ have at least one common neighbor from $V_i$ and another from $V_j$ in $\overline{G}$. By Lemma \ref{rvddifcolor}, the vertex set $T$ is rainbow in $\overline{G}$. Thus, $rvd(\overline{G})\geq n-12$ and $rvd(G)+rvd(\overline{G})\geq \big\lceil\frac{n}{3}\big\rceil+n-12\geq n-7$.

\textbf{Case 3}. There is only one $V_i$ with $|V_i|=3$ and $|V_s|\leq 2$ for $s\in [k]\setminus \{i\}$.

We have $rvd(\overline{G})\geq rvd(G)\geq \frac{n-3}{2}+1=\frac{n-1}{2}$. So, $rvd(G)+rvd(\overline{G})\geq n-1$.
\end{proof}

\begin{thm}\label{rvdNGaddn-5}
Let $G$ and $\overline{G}$ be connected graphs of order $n\geq 24$. Then $n-5 \leq rvd(G)+rvd(\overline{G})\leq 2n$ and the upper bound is sharp.
\end{thm}
\begin{proof}
Let $rvd(G)=k$ and $\{V_1,V_2,\cdots,V_k\}$ be the set of color classes of an rvd-coloring of $G$. Then for any triple $\{v_1,v_2,v_3\}\subseteq V_i$, where $i\in [k]$, let $S=V(\overline{G})\setminus\{v_1,v_2,v_3\}$. Since $m_G(v_1,v_3)\leq 1$, we have $v_3$ is adjacent to at least $|S-N_{\overline{G}}(v_1)|-1$ vertices of the vertex set $S-N_{\overline{G}}(v_1)$ in $\overline{G}$. Since $m_G(v_2,v_3)\leq 1$, we have that $v_3$ is adjacent to at least $|S-N_{\overline{G}}(v_2)|-1$ vertices of the vertex set $S-N_{\overline{G}}(v_2)$ in $\overline{G}$. Since $m_G(v_1,v_2)\leq 1$, we obtain $|(S-N_{\overline{G}}(v_1))\cap (S-N_{\overline{G}}(v_2))|\leq 1$. If $d_{\overline{G}}(v_1)<\frac{n+2}{2}$ and $d_{\overline{G}}(v_2)< \frac{n+2}{2}$, then
\begin{align*}
d_{\overline{G}}(v_3)&\geq |S-N_{\overline{G}}(v_1)|+|S-N_{\overline{G}}(v_2)|-3\\
      &=2n-9-d_{\overline{G}}(v_1)-d_{\overline{G}}(v_2)\\
      &> n-11\\
      &\geq \frac{n+2}{2}.
\end{align*}
So, for $\overline{G}$ there is at least one vertex with degree more than $\frac{n+2}{2}$ in $\{v_1,v_2,v_3\}$. Let $T$ be the set of vertices with degrees larger than $\frac{n+2}{2}$ in $\overline{G}$. Then we have $|T|\geq \sum_{i\in [k]}(|V_i|-2)=n-2k$. For any two vertices $x$ and $y$ in $T$, $d_{\overline{G}}(x)+d_{\overline{G}}(y)\geq n+2$. So, we have that $T$ is rainbow by Lemma \ref{rvddifcolor}. Thus, $rvd(\overline{G})\geq n-2k$.
When $k\leq 5$, we have $rvd(G)+rvd(\overline{G})\geq n-k\geq n-5$. Now consider $k\geq 6$. Since $n\geq 24$, we have $rvd(G)+rvd(\overline{G})\geq n-4$ for the Case $2$ and Case $3$ of Theorem \ref{rvdNGaddn-7}. For the Case $1$ of Theorem \ref{rvdNGaddn-7}, we have $rvd(G)+rvd(\overline{G})\geq 6+n-10\geq n-4$.

The upper bound is sharp, which can be achieved by the graph $G$ with order $n=4k+t$ ($k\geq 6$ and $t=0,1,2,3$), described in Theorem \ref{rvdNGmulti}.
\end{proof}

In fact, we think that the lower bound of $rvd(G)+rvd(\overline{G})$ could be improved further. When $rvd(G)=1$, we have $rvd(G)+rvd(\overline{G})\geq n$ for $n\geq 8$ by Lemma \ref{rvd1comple}. So, we pose the following conjecture for further study.

\begin{conj}
Let $G$ and $\overline{G}$ be nontrivial connected graphs of order $n\geq 8$. Then $rvd(G)+rvd(\overline{G})\geq n$.
\end{conj}

\end{document}